\documentclass{amsart}

\usepackage{graphicx}
\usepackage{mathrsfs}
\usepackage{appendix}

\numberwithin{equation}{section}

\newtheorem{prop}{Proposition}[section]
\newtheorem{theo}[prop]{Theorem}
\newtheorem{lemm}[prop]{Lemma}
\newtheorem{coro}[prop]{Corollary}
\newtheorem{exam}[prop]{Example}

\newtheorem{claim}[prop]{Claim}

\theoremstyle{definition}
\newtheorem{defi}[prop]{Definition}

\newtheorem{rema}[prop]{Remark}

\newcommand{\NN}{\mathbb{N}}

\newcommand{\RR}{\mathbb{R}}

\newcommand{\cS}{\mathcal S}

\newcommand{\sF}{\mathscr{F}}
\newcommand{\sH}{\mathscr{H}}

\DeclareMathOperator{\Index}{index}
\DeclareMathOperator{\Span}{span}

\DeclareMathOperator{\supp}{supp}

\DeclareMathOperator{\Hess}{Hess}

\DeclareMathOperator{\Vol}{Vol}
\DeclareMathOperator{\vol}{vol}

\DeclareMathOperator{\dist}{dist}

\DeclareMathOperator{\ind}{index}
\DeclareMathOperator{\nul}{nullity}

\DeclareMathOperator{\Ric}{Ric}

\newcommand{\pa}[2]{\frac{\partial #1}{\partial #2}}
\newcommand{\paop}[1]{\pa{}{#1}}

\newcommand{\bangle}[1]{\left\langle #1 \right\rangle}

\numberwithin{equation}{section}

\begin{document}

\title{Index and topology of minimal hypersurfaces in $\RR^n$}

\author{Chao Li}
\address{Department of Mathematics, Stanford University}
\curraddr{Department of Mathematics,
Stanford University, CA, 94305}
\email{rchlch@stanford.edu}



\date{Feburary, 2016}

\begin{abstract}
In this paper, we consider immersed two-sided minimal hypersurfaces in $\RR^n$ with finite total curvature. We prove that the sum of the Morse index and the nullity of the Jacobi operator is bounded from below by a linear function of the number of ends and the first Betti number of the hypersurface. When $n=4$, we are able to drop the nullity term by a careful study for the rigidity case. Our result is the first effective Morse index bound by purely topological invariants, and is a generalization of \cite{li2002minimal}. Using our index estimates and ideas from the recent work of Chodosh-Ketover-Maximo \cite{chodosh2015minimal}, we prove compactness and finiteness results of minimal hypersurfaces in $\RR^4$ with finite index.
\end{abstract}

\maketitle

\section{Introduction}
Minimal hypersurfaces of the Euclidean spaces $\RR^n$ are critical points of the area functional. The Jacobi operator from second variation of area functional gives rise to the Morse index of the minimal surface. In Euclidean spaces $\RR^n$, the second variation formula for a two-sided minimal hypersurface $\Sigma$ is given by
\[Q(f,f)=\int_{\Sigma}|\nabla f|^2-|A|^2f^2.\]
It induces a second order elliptic operator
\[J(f)=\Delta f+|A|^2f,\]
where $|A|^2$ is the sum of square of principal curvatures, and $f$ is a compactly supported smooth function representing the normal variation. The Morse index of a compact subset $K\cap \Sigma$ is defined to be the number of negative eigenfunctions of $J$ with Dirichlet boundary condition. By the domain monotonicity of eigenvalues, when $K_1\subset K_2$, $\ind(K_1\cap \Sigma)\le \ind(K_2\cap \Sigma)$. Hence we may define the Morse index of $\Sigma$ to be $\lim_{R\rightarrow \infty}\ind(B_R(0)\cap \Sigma)$. This limit exists and may be infinity.

The classical Bernstein theorem \cite{bernstein1927geometrisches} asserts that an entire solution to the minimal surface equation in $\RR^2$ must be affine. Later, it was proved by Fischer-Colbrie-Schoen \cite{fischer1980structure}, do Carmo-Peng \cite{do1979stable} and Pogorelov \cite{pogorelov81stability} that the plane is the only stable (index $0$) minimal surface in $\RR^3$. If we allow positive Morse index, there are lots of examples of complete immersed minimal surfaces in $\RR^3$. In \cite{costa82example}, \cite{hoffman1985embedded} and \cite{hoffman1990embedded}, the authors constructed embedded minimal surfaces of genus $g$ with any $g\ge 1$. The index of a genus $g$ Costa-Hoffman-Meeks surfaces is $2g+3$, by \cite{natayani1992morse} and \cite{morabito2009index}. Another example of an immersed minimal surface with finite topology is the Jorge-Meeks surface \cite{jorge1983topology}: for any integer $r\ge 3$, there is an immersed simply connected minimal surface with $r$ catenoidal ends. The index of a Jorge-Meeks surface with $r$ ends is $2r-3$ \cite{montiel1991schrodinger}. These examples indicate a good control of the topology of a minimal surface in $\RR^3$ by its Morse index.

The relationship between the topology and Morse index of a minimal surface has been studied by many authors. From the work of Fischer-Colbrie \cite{fischer1985complete}, we know that if minimal hypersurface in a $3$ dimensional manifold has finite Morse index, then outside a compact part the surface is stable. \cite{choe1990index}, \cite{ros2006one} and \cite{otisdavi2014index} prove that the index of a minimal surface in $\RR^3$ is bounded from below by a linear function of the number of ends and the genus. \cite{otisdavi2014index} also summarized various known results connecting the index and topology of minimal surfaces in $\RR^3$ with finite total curvature.

Similar study for minimal hypersurfaces in higher dimension has been limited, due to the lack of concrete examples and the inavailability of complex analytic tools. Examples of minimal hypersurfaces have been constructed in, for instances, \cite{Choe_onthe} and \cite{Choe_schwarz}, but none of which is complete and has finite total curvature. To our best knowledge, the only examples of complete minimal hypersurfaces with finite total curvature are the plane and the catenoid. In \cite{cao1997structure}, Cao, Shen and Zhu proved that for all $n\ge 4$, complete two-sided stable minimal hypersurfaces have at most one end. Later Shen and Zhu \cite{shen1998stable} proved that any complete stable minimal hypersurface in $\RR^n$ with finite total curvature must be a plane. For minimal hypersurfaces with positive Morse index, Tam and Zhou \cite{tam2009stability} showed that the high dimensional catenoid has index $1$. R. Schoen proved in \cite{schoen1983} that the catenoid is the only connected minimal hypersurface with two regular ends. Li and Wang \cite{li2002minimal} proved that finite index implies finitely many of ends. However, their result did not give an explicit control of the number of ends by the index of the minimal hypersurface.

It was pointed out by \cite{shoen1976harmonic} that the existence of an $L^2$ harmonic $1$-form violates stability. This was utilized by Cao-Shen-Zhu in \cite{cao1997structure} and by Li-Wang in \cite{li2002minimal}. Later Mei and Xu in \cite{mei2001} pointed out that if the minimal hypersurface has $k$ ends, then there exists a $k-1$ dimensional space of $L^2$ harmonic $1$-forms. \cite{tanno19962} also investigated the connection between $L^2$ harmonic $2$-forms and the stability in low dimensions.

In this paper, we combine an idea of Savo \cite{savo2010index} with the harmonic $1$-form technique discussed above to get an effective estimate of certain topological invariants and the index of minimal hypersurfaces in $\RR^n$. In fact, we can prove:

\begin{theo}\label{rough}
Let $\Sigma^{n-1}$ be a complete connected two-sided minimal hypersurface in $\RR^n$, $n\ge 4$. Suppose that $\Sigma$ has finite total curvature, that is, $\int_\Sigma |A|^{n-1}$ is finite. Then we have
\[\ind(\Sigma)+\nul(\Sigma) \ge \frac{2}{n(n-1)}(\#\textrm{ends}+b_1(\bar{\Sigma})-1),\]
where $\nul(\Sigma)$ is the dimension of the space of $L^2$ solutions of the Jacobi operator, and $b_1(\bar{\Sigma})$ is the first Betti number of the compactification of $\Sigma$.
\end{theo}

By \cite{tysk1989finiteness}, when $4\le n\le 7$ and $\Sigma^{n-1}\subset \RR^n$ has Euclidean volume growth (that is, $\lim_{R\rightarrow \infty}\vol(\Sigma\cap B_R(0))/R^{n-1}<\infty$), then $\Sigma$ has finite total curvature if and only if it has finite index. Therefore

\begin{coro}
Let $\Sigma^{n-1}$ be a complete connected two-sided minimal hypersurface in $\RR^n$, $4\le n\le 7$. Suppose that $\Sigma$ has Euclidean volume growth. Then
\[\ind(\Sigma)+\nul(\Sigma) \ge \frac{2}{n(n-1)}(\#\textrm{ends}+b_1(\bar{\Sigma})-1).\]
\end{coro}

Through a careful study of the rigidity case, we are able to get rid of the nullity term with the extra assumption that the dimension $n=4$, or that there exists one point on $\Sigma$ where the principal curvatures are all distinct. Namely, we have
\begin{theo}\label{precise}
Let $\Sigma^{n-1} \subset \RR^{n}$ be a complete connected two-sided minimal hypersurface with finite total curvature. Suppose that $n=4$, or that there exists a point on $\Sigma$ where all the principal curvatures are distinct. Then
\[\ind(M)\ge \frac{2}{n(n-1)} (\# \textrm{ends}+b_1(\bar{\Sigma}))-\frac{4}{n}.\]
\end{theo}

The assumptions of Euclidean volume growth or finite total curvature in the previous two theorems are natural. In \cite{tysk1989finiteness}, Tysk proved that all minimal hypersurfaces with finite total curvature must be regular at infinity (see \cite{schoen1983} for the definition of regular at infinity). That is, at each end the surface is a graph over some plane of a function decaying like $C|x|^{-n+2}$. This precise large scale behavior of each end enables us to perform a more precise analysis.

Our theorem has some interesting applications in the study of minimal hypersurfaces in Euclidean space. For example, complete minimal hypersurfaces of index one is one of the most natural objects occured in geometric variational problems. In \cite{ros1989}, L\'opez and Ros proved that the catenoid and the Enneper surface are the only index one complete connected two-sided immersed minimal surface in $\RR^3$. It is unknown if the catenoid is the only index $1$ complete embedded minimal hypersurface in the Euclidean space. By \cite{schoen1983}, we know that if a minimal hypersurface has two regular ends, then it is a catenoid. Theorem \ref{precise} is not strong enough to conclude this. However, we do have the following properties of the space of index $1$ minimal hypersurfaces in $\RR^4$.

\begin{theo}\label{compactnessofindex1}
The space of complete connected embedded two-sided index $1$ minimal hypersurfaces $\Sigma^3\subset\RR^4$ with Euclidean volume growth, normalized such that $|A_\Sigma|(0)=\max |A|_\Sigma=1$, is compact in the smooth topology.
\end{theo}

\begin{theo}
There exists a constant $R_0$ such that the following holds: for any complete connected embedded two-sided minimal hypersurface $\Sigma\subset \RR^4$ with finite total curvature and index $1$, normalized so that $|A_\Sigma|(0)=\max|A_\Sigma|=1$, $\Sigma$ is a union of minimal graphs in $\RR^4-B_R(0)$.
\end{theo}

Such property is not expected for a family of minimal hypersurfaces with larger index bound, as illustrated by the following example.

\begin{exam}\label{costa}
Let $C_0$ be the genus $2$ Costa-Hoffman-Meeks surface with $3$ ends, one planar end and two catenoidal ends behaving like $\log |x|, -\log |x|$ near infinity. It is known that there is a family of deformed surfaces $C_t$ with three catenoidal ends whose growth rate near infinity is approximately $a_t\log |x|, \log |x|, b_t\log |x|$, with $a_t>0>b_t, a_t+b_t+1=0$. The surface $C_t$ qualitatively looks like three surfaces, each with one catenoidal end, joined by three catenoidal necks. The curvature of the surface $C_t$ is maximized at the three catenoidal necks. Now if we normalize each $C_t$ to $C'_t$, with $|A_{C'_t}|$ maximized at $|A_{C'_t}|(0)=1$ where $0$ is on one of the three necks, then other necks of $C'_t$ drifts to infinity as $t$ goes to infinity. In particular, for any $R>0$, there is $C'_t$ which is not graphical outside $B_R(0)$. However, the family $C'_t$ have uniformly bounded index.
\end{exam}

The second application is the finiteness of diffeomorphism types of minimal hypersurfaces in $\RR^4$ with Euclidean volume growth and bounded index. Using theorem \ref{precise} and ideas from the recent work of Chodosh-Ketover-Maximo \cite{chodosh2015minimal}, we are able to get the following:

\begin{theo}
There exists $N=N(I)$ such that there are at most $N$ mutually non-diffeomorphic complete embedded minimal hypersurfaces $\Sigma^3$ in $\RR^4$ with Euclidean volume growth and $\ind(\Sigma)\le I$.
\end{theo}

It would be interesting to see in more generality how the index of a minimal hypersurface in $\RR^n$ depends on its topological invariants. It is conjectured that a similar statement as in Theorem \ref{precise} should hold for $4\le n \le 7$. Even in dimension $4$, we believe that the inequality of Theorem \ref{precise} is not optimal. For example, it does not answer the question of whether the higher dimensional catenoid is the only minimal hypersurface in Euclidean space of index $1$. These are interesting questions to investigate in future.

The author would like to express his most sincere gratitude to his advisors, Rick Schoen and Brian White, for bringing this question to his attention and for several enlightening discussions. He also wants to thank Robert Bryant and Jesse Madnick for their insights in the rigidity discussion, and David Hoffman for a careful description of the Costa-Hoffman-Meeks surfaces. Special thanks go to the referee for many illustrating suggestions.

\section{Spectral properties of minimal hypersurface with finite total curvature}

We start by revisiting the following classical result of Fischer-Colbrie.

\begin{theo}[\cite{fischer1985complete}]\label{fischer}
Let $\Sigma^{n-1}\subset\RR^n$ be a complete two-sided minimal surface of index $k$. Then there exist $L^2$ orthonormal eigenfunctions $f_1,\ldots,f_k$ of the Jacobi operator $J$ associated to negative eigenvalues, such that for any compactly supported smooth function $f$ on $\Sigma$ that is $L^2$ orthogonal to $f_1,\ldots,f_k$, $Q(f,f)\ge 0$.
\end{theo}

\begin{rema}
	In \cite{fischer1985complete}, the above theorem is stated for minimal surfaces in $3$-manifolds. However, the same proof generalizes for complete two-sided minimal hypersurfaces $\Sigma^{n-1}$ in $\RR^n$ without much difficulty.
\end{rema}

Let us now recall the definition for a minimal hypersurface to be regular at infinity.

\begin{defi}[\cite{schoen1983}]
Suppose $n\ge 4$. A minimal hypersurface $\Sigma^{n-1}\subset \RR^n$ is regular at infinity, if outside a compact set, each connected component of $\Sigma$ is the graph of a function $u$ over a hyperplane $P$, such that for $x\in P$,
\[|x|^{n-3}|u(x)|+|x|^{n-2}|\nabla u(x)|+|x|^{n-1}|\nabla^2 u(x)|\le C,\]
where $C$ is some constant.
\end{defi}

In order to perform a more careful rigidity analysis, we use the extra condition that the minimal hypersurface $\Sigma$ has finite total curvature. By a result of M. Anderson $(n\ge 4)$ \cite{anderson1984compactification}, finite total curvature implies that $\Sigma^{n-1}$ is diffeomorphic to a compact manifold $\bar{\Sigma}$ minus finitely many points (in one-to-one correspondence to its ends). In fact, we have:
\begin{prop}[\cite{tysk1989finiteness}]
Suppose $n\ge 4$, $\Sigma^{n-1}$ in $\RR^n$ is a complete immersed minimal hypersurface with finite total curvature. Then $\Sigma$ is regular at infinity.
\end{prop}

For our purposes, we use the fact that if $\Sigma$ has finite total curvature, then $|A|$ is bounded on $\Sigma$, and the induced metric on $\Sigma$ tends to the Euclidean metric near infinity in the $C^2$ sense.

\begin{prop}\label{generalizedfischer}
Let $\Sigma^{n-1}\subset\RR^n$ be a complete minimal hypersurface with index $k$ that is regular at infinity, and let $f_1,\ldots,f_k$ be $k$ $L^2$ orthonormal eigenfunctions with negative eigenvalue given by theorem \ref{fischer}. Then for any function $f\in C^\infty(\Sigma)\cap W^{1,2}(\Sigma)$ that is $L^2$ orthogonal to $f_1,\ldots,f_k$, $Q(f,f)\ge 0$. Moreover if $Q(f,f)=0$, then $f$ is a solution of $J(f)=0$.
\end{prop}

\begin{proof}
We first observe that each $f_j$ is in fact in $W^{1,2}$. Indeed, $f_j$ is a solution of $\Delta f_j+|A|^2 f_j=\lambda _j f_j$. Since $\Sigma$ is regular at infinity, the operator $\Delta_{\Sigma}$ is a uniformly elliptic operator, and $|A|^2$ is bounded. Therefore by a covering argument and elliptic estimates, we have $\|\nabla f_j\|_{L^2(\Sigma)}\le C\|f_j\|_{L^2(\Sigma)}<\infty$.

The first statement follows from a standard cutoff argument. Now let us assume $Q(f,f)=0$. We will prove $Q(f,g)=0$ for any $g\in W^{1,2}(\Sigma)$.

Let's first assume $g$ is a compactly supported smooth function that is $L^2$ orthogonal to $f_1,\ldots,f_k$. Take a large $R>0$ so that $\supp(g)$ is contained in $B_R(0)\cap \Sigma$. Choose a cutoff function $\varphi$  which is $1$ on $B_R(0)\cap \Sigma$ and $0$ outside $B_{2R}\cap \Sigma$. Denote $f_t=\varphi (f+tg+c_1f_1+\ldots+c_kf_k)$, where $c_1,\ldots,c_k$ are properly chosen real numbers such that $f_t$ is $L^2$ orthonogal to $f_1,\ldots,f_k$. Since $g$ is $L^2$ orthogonal to $f_1,\ldots,f_k$, each $c_j$ is independent of $t$. By theorem \ref{fischer}, we have $Q(f_t,f_t)\ge 0$ for any real number $t$. Now
\[
\begin{aligned}
Q(f_t,f_t)&=t^2Q(\varphi g,\varphi g)+2t\left(Q(\varphi f,\varphi g)+\sum_j Q(\varphi g,\varphi c_jf_j)\right)\\
            &+\sum_{i,j}c_ic_j Q(\varphi f_i,\varphi f_j)+Q(\varphi f,\varphi f)+2\sum_j Q(\varphi f, \varphi c_jf_j).
\end{aligned}
\]

Note that $Q(\varphi g, \varphi f_j)=0$. Since $Q(f_t,f_t)\ge 0$ for all $t$, we conclude
\[Q(\varphi f,\varphi g)^2\le Q(\varphi g,\varphi g)\cdot \left(\sum_{i,j}c_ic_j Q(\varphi f_i,\varphi f_j)+Q(\varphi f,\varphi f)+2\sum_j Q(\varphi f, \varphi c_jf_j)\right).\]
Let $R\rightarrow \infty$. Then $c_j\rightarrow 0$ and $Q(\varphi f,\varphi f)\rightarrow Q(f,f)=0$, $Q(\varphi f,\varphi g)\rightarrow Q(f,g)$. Therefore we get $Q(f,g)^2\le 0$, so $Q(f,g)=0$.

Now if $g\in C^\infty(\Sigma)\cap W^{1,2}$ is not compactly supported but is still $L^2$ orthonogal to $f_1,\ldots,f_k$, then $g$ can be approximated by compactly supported smooth functions that are $L^2$ orthogonal to $f_1,\ldots,f_k$. This implies $Q(f,g)=0$.

Next we show $Q(f,f_j)=0$. We use the fact that each $f_j$ is in $W^{1,2}$, so it is a weak limit of a sequence of eigenfunctions of $J$ on $B_{R_i}(0)\cap \Sigma$. The statement now follows from a cutoff argument similar to the one before. Hence $Q(f,g)=0$ for $g$ in the span of $f_1,\ldots,f_k$ and in its $L^2$ orthogonal complement, therefore $Q(f,g)=0$ for each $g$ in $W^{1,2}(\Sigma)$.
\end{proof}

\section{The space of bounded harmonic functions on $\Sigma$}
The statement in this section can be found in \cite{cao1997structure} and \cite{mei2001}. We include the proof here because bounded harmonic functions are essential in the construction of test functions for the stability operator.

\begin{prop}[\cite{cao1997structure},\cite{mei2001}]
Let $n\ge 4$ and $\Sigma^{n-1}$ be a complete minimal hypersurface in $\RR^n$ with $k$ ends. Then the are $k$ linearly independent bounded harmonic functions with finite Dirichlet energy.
\end{prop}

\begin{proof}
When $k=1$ the constant function is harmonic. Suppose $k\ge 2$. Suppose for some compact domain $K$, $\Sigma-K=E_1\cup E_2\cup\ldots \cup E_k$, where $E_1,\ldots,E_k$ are the $k$ ends. For $R$ large enough, $\Sigma\cap B_R(0)$ has $k$ boundary components. Solve the Dirichlet problem
\begin{equation*}
    \begin{cases}
    \Delta f_{i,R} &= 0 \quad \textrm{ in $\Sigma \cap B_R(0)$,} \\
    f_{i,R} &= \delta_{ij} \quad \textrm{on $\partial B_R(0)\cap E_j$, $j=1,\ldots,k$.}
    \end{cases}
\end{equation*}

By the maximum principle $0<f_{i,R}<1$ in $B_R(0)\cap \Sigma$. Using Schauder theory we get a uniform bound on $|f_{i,R}|_{C^{2,\alpha}(K)}$ for each compact $K\subset B_R(0)$. Therefore we may use Arzela-Ascoli to get a subsequence $\{f_{i,R}\}_{R}$ converging to $f_i$ in $C^{2,\beta}$ ($\beta<\alpha$). For $R_1<R_2$, the function $f_{i,R_1}$ can be extended with constant value to a function on $B_{R_2}(0)\cap \Sigma$. Since harmonic functions minimize Dirichlet energy, $\int |\nabla f_{i,R_1}|^2>\int |\nabla f_{i,R_2}^2|$. Therefore the function $f_i$ is a bounded harmonic function with finite Dirichlet energy.

Next we prove $f_i$ is not a constant function. Suppose the contrary. The function $f_i(1-f_i)$ is in $W_0^{1,2}(\Sigma)$, hence by the Michael-Simon Sobolev inequality \cite{michaelsimonsobolev}, we have for $\kappa=\frac{n}{n-2}$,
\[\left(\int |f_i(1-f_i)|^{2\kappa}\right)^{1/\kappa} \le C \int |\nabla (f_i(1-f_i))|^2 \le 4C\int |\nabla f_i|^2=0.\]
Therefore $f_i$ is identically $0$ or $1$. Without loss of generality we assume $f_i\equiv 1$ (otherwise consider $1-f_i$ instead). Choose some $l\ne i$. Now take any smooth function $\varphi$ which is identically $1$ on $E_l$, $0$ on all other ends. Then $f_{i,R}\varphi$ is compactly supported. By the Michael-Simon Sobolev inequality and the fact that $\nabla \varphi$ is compactly supported,
\begin{equation*}
    \begin{split}
        \left(\int |f_{i,R}\varphi|^{2\kappa}\right)^{1/\kappa} &\le C \int |\nabla (f_{i,R}\varphi)|^2\\
                                                                &\le 2C \int |\nabla f_{i,R}|^2 + |\nabla \varphi|^2\\
                                                                &<\infty
    \end{split}
\end{equation*}
Letting $R\rightarrow\infty$ we see that $\int f_i \varphi=\lim _{R\rightarrow \infty}\int f_{i,R}\varphi<\infty$, contradicting the fact that the $l$-th end $E_l$ has infinite volume.

By similar reasoning, the functions $f_1,\ldots,f_k$ are linearly independent. Otherwise we would have $u=c_1f_1+\ldots+c_kf_k=0$. However, $u$ is the $C^{2,\beta}$ limit of some compactly supported harmonic functions taking $c_1,\ldots,c_k$ as boundary values. An argument similar to the one before shows that such a $u$ cannot be constant.
\end{proof}

\begin{rema}
	As a technical remark, in \cite{otisdavi2014index}, the authors utilized the existence of harmonic $1$-forms in a similar way. The major difference is that, harmonic functions on $\Sigma^{n-1}$ tending to constant on each end have finite Dirichlet energy if and only if $n\ge 4$. Therefore in \cite{otisdavi2014index} the authors have to use a weighted space rather than $L^2$.
\end{rema}

\section{Proof of the main theorem}
In this section we prove Theorem \ref{rough}. We start by collecting a family of $L^2$ harmonic $1$-forms.

\begin{prop}
Let $\Sigma^{n-1}$ be a complete minimal hypersurface in $\RR^n$ that is regular at infinity. Suppose $\Sigma$ has $k$ ends. Then there are $k+b_1(\bar{\Sigma})-1$ linearly independent closed $L^2$ harmonic $1$-forms on $\Sigma$ with finite Dirichlet energy.
\end{prop}
\begin{proof}
Take the functions $f_1,\ldots,f_k$ constructed in section 3. Their differentials $df_1,\ldots,df_k$ are harmonic since $d$ and $\Delta$ commute. We prove that $\Span\{df_1,\ldots,df_k\}$ is $k-1$ dimensional. The function $f_1+\ldots+f_k $ is the limit of a sequence of harmonic functions with boundary values $1$. By the maximum principle, each harmonic function in the sequence is identically $1$. So $f_1+\ldots+f_k$ is also the constant function $1$. We see $df_1+\ldots+df_k=0$. Suppose $df_1,\ldots,df_j$ are linearly dependent for some $j<k$. Then $c_1df_1+\ldots+c_jdf_j=0$. Therefore $c_1f_1+\ldots+c_jf_j$ is a constant function on $\Sigma$. Combine $f_1+\ldots+f_k=1$, we get a nontrivial linear combination of $f_1,\ldots,f_k$ that equals $0$, contradicting the linear independence of $f_1,\ldots,f_k$.

If $b_1(\bar{\Sigma})>0$, then we have ${b_1(\bar{\Sigma})}$ linearly independent closed non-exact harmonic $1$-forms $\eta_1,\ldots,\eta_{b_1(\bar{\Sigma})}$. Then the set $\{df_1,\ldots,df_{k-1},\eta_1,\ldots,\eta_{b_1(\bar{\Sigma})}\}$ is a set of $k+b_1(\bar{\Sigma})-1$ linearly independent closed harmonic $1$-forms on $\Sigma$.
\end{proof}

Now let us fix some notations. For any minimal hypersurface $\Sigma^{n-1}$ in $\RR^n$, let $\bar{\nabla}$ be the Euclidean connection on $\RR^n$ and $\nabla$ be the Levi-Civita connection of the induced metric on $\Sigma$. Denote the Hodge Laplacian on $p$-forms by $\Delta=-(d\delta+\delta d)$. Suppose $\Sigma$ is two-sided with a unit normal vector $\nu$. Take two vector fields $X,Y$ on $\Sigma$. Let $S$ be the shape operator defined by $S(X)=-\bar{\nabla}_X \nu$, and let $A$ be the second fundamental form defined by $A(X,Y)=\bangle{S(X),Y}$. For two parallel vectors $\bar{W},\bar{V}$ in $\RR^n$, let $W,V$ be their projection on $\Sigma$. Let $\omega$ be a harmonic $1$-form on $\Sigma$ and $\xi$ its dual vector field. With these notations, we have

\begin{lemm}[Lemma 2.4 of \cite{savo2010index}]\label{calculation}
\begin{enumerate}
    \item $\nabla_X W=\bangle{\bar{W},\nu}S(X)$,
    \item $\nabla \bangle{\bar{W},\nu}=-S(W)$,
    \item $\Delta \bangle{\bar{W},\nu}=-|A|^2\bangle{\bar{W},\nu}$,
    \item $\Delta \bangle{W,\xi}=2\bangle{S(W),S(\xi)}-2\bangle{\bar{W},\nu}\bangle{A,\nabla \xi}$,
    \item $\Delta\left(\bangle{\bar{V},\nu}\bangle{W,\xi}\right)=-|A|^2\bangle{\bar{V},\nu}\bangle{W,\xi}+\alpha(V,W,\xi)$, where $\alpha$ is a $(0,3)$ tensor, defined by
        \[
        \begin{split}
            \alpha(V,W,\xi)=& -2\bangle{\bar{V},\nu}\bangle{S(W),S(\xi)}-2\bangle{\bar{W},\nu}\bangle{S(V),S(\xi)}\\
                            & +2\bangle{\bar{V},\nu}\bangle{\bar{W},\nu}\bangle{A,\nabla\xi}-2\bangle{\nabla_{S(V)}\xi,W}.
        \end{split}
        \]
\end{enumerate}
\end{lemm}

\begin{proof}
$(1)$ to $(3)$ are standard facts about minimal hypersurfaces in Euclidean spaces. To prove $(4)$, take $\eta$ to be the dual $1$-form of the vector field $W$. Then $\omega$ and $\eta$ are both harmonic $1$-forms. Denote the connection Laplacian by $\nabla^2$. We have
\[\Delta \bangle{W,\xi}=\nabla^2\bangle{\eta,\omega}=\bangle{\nabla^2 \eta, \omega}+\bangle{\eta, \nabla^2 \omega}+2\bangle{\nabla \eta,\nabla \omega}.\]
By the Bochner identity for the harmonic $1$-forms $\eta$ and $\omega$, we have $0=\Delta \omega=\nabla^2 \omega-\Ric(\xi,\cdot)$, $0=\Delta \eta=\nabla^2 \eta-\Ric(W,\cdot)$. Therefore $\bangle{\nabla^2 \eta,\omega}=\bangle{\nabla^2 \omega, \eta}=\Ric(W,\xi)$. Using the Gauss equation for $\Sigma^{n-1}$ in $\RR^n$, we see that $\Ric(X,\xi)=-\bangle{S(X),S(\xi)}$. Also by $(1)$, $\nabla W=\bangle{\bar{W},\nu}A$. So $\bangle{\nabla \eta,\nabla \omega}=\bangle{\nabla W,\nabla \xi}=\bangle{\bar{W},\nu}\bangle{A, \nabla \xi}$. This proves $(4)$.

To prove $(5)$, we see that
\[\Delta\left(\bangle{\bar{V},\nu} \bangle{W,\xi}\right)=\left(\Delta\bangle{\bar{V},\nu}\right)\bangle{W,\xi}+\bangle{\bar{V},\nu}\left(\Delta \bangle{W,\xi}\right)+2\bangle{\nabla\bangle{\bar{V},\nu},\nabla\bangle{W,\xi}}.\]
We use $(3)$ and $(4)$ to simplify the first two terms. For the third term, we have
\[
\begin{split}
    \bangle{\nabla\bangle{\bar{V},\nu},\nabla\bangle{W,\xi}}&=\bangle{S(V),\nabla \bangle{W,\xi}}\\
                                                            &=\bangle{\nabla_{S(V)}W,\xi}+\bangle{W,\nabla_{S(V)}\xi}\\
                                                            &=\bangle{\bar{W},\nu}\bangle{S^2(V),\xi}+\bangle{\nabla_{S(V)}\xi,W}\\
                                                            &=\bangle{\bar{W},\nu}\bangle{S(V),S(\xi)}+\bangle{\nabla_{S(V)}\xi,W},
\end{split}
\]
where the first equality is true by $(2)$, and the third equality by $(1)$, the fourth equality by the fact that $S$ is symmetric.

Using the above equality and $(3)$, $(4)$ we get $(5)$.
\end{proof}

Now we are ready to prove Theorem \ref{rough}. Take $x_1,\ldots,x_n$ to be the standard coordinates of $\RR^n$. The vector fields $\bar{V_1}=\paop{x_1},\ldots,\bar{V_n}=\paop{x_n}$ are parallel vector fields in $\RR^n$. Their projections onto $\Sigma$ are denoted by $V_1,\ldots,V_n$. Define the vector fields $X_{ij}=\bangle{V_i,\nu}V_j-\bangle{V_j,\nu}V_i$. For a harmonic $1$-form $\omega$ on $\Sigma$ dual to a vector field $\xi$, define the functions $f_{\omega,ij}=\bangle{\omega,X_{ij}}=\bangle{\bar{V_i},\nu}\bangle{V_j,\omega}-\bangle{\bar{V_j},\nu}\bangle{V_i,\omega}$ for $1\le i<j\le n$. It is clear that $f_{\omega,ij}$ is in $L^2$. By lemma \ref{calculation},
\begin{equation}\label{identity1}
\Delta f_{\omega,ij}=-|A|^2 f_{\omega,ij} -2\bangle{\nabla_{S(V_i)}\xi,V_j}+2\bangle{\nabla_{S(V_j)}\xi,V_i}.
\end{equation}

Regarding the integrability of $f_{\omega,ij}$, we have the following
\begin{lemm}
	$f_{\omega, ij}$ defined above is in $W^{1,2}(\Sigma)$.
\end{lemm}

\begin{proof}
Fix an $\omega$ and a pair of $(i,j)$ as above. Note that $\omega$ is an $L^2$ one form. For the simplicity of notations, let us denote $f_{\omega,ij}$ by $f$ and $-|A|^2f_{\omega,ij}-2\bangle{\nabla_{S(V_i)}\xi,V_j}+2\bangle{\nabla_{S(V_j)}\xi,V_i}$ by $g$. Since $|\omega|\in L^2(\Sigma)$, we know that $f\in L^2(\Sigma)$. Also since $|A|$ is bounded on $\Sigma$, $g\in L^2(\Sigma)$. Fix a point $p\in \Sigma$. Take a cutoff function $\eta$ which is supported in the Euclidean ball $B_2(p)$ and is identically $1$ in $B_1(p)$, and satisfies $|\nabla \eta|\le 2$. Since $\Delta f=g$, we have
\[\int_{B_2(p)\cap \Sigma}f\Delta f\eta=\int_{B_2(p)\cap \Sigma}\eta fg.\]
Integrate by parts and using the fact that $\eta$ is zero on $\partial B_2(p)$, we have that
\[\int_{B_2(p)\cap \Sigma}\eta|\nabla f|^2=-\int_{B_2(p)\cap\Sigma}\eta gf-\int_{B_2(p)\cap\Sigma}f\nabla f\cdot\nabla \eta.\]

By Cauchy-Schwartz, we have
\[\int_{B_1(p)\cap\Sigma}|\nabla f|^2\le \|f\|_{L^2(\Sigma)}\|g\|_{L^2(B_2(p))}+\int_{B_{2}(p)\cap\Sigma}|f\nabla
f \cdot \nabla\eta|.\]
Let $\epsilon>0$ be a small number chosen later and using arithmatic-geometric mean inequality, we deduce that
\[\int_{B_1(p)\cap\Sigma}|\nabla f|^2\le \|f\|_{L^2(\Sigma)}\|g\|_{L^2(B_2(p))} +\left(\epsilon \int_{B_2-B_1}|\nabla f|^2+\frac{1}{\epsilon}\int_{B_2-B_1}|f|^2\right).\]

Let $C$ be the cube inscribed to $B_1(p)$. Then $B_2(p)$ is contained in $3C$, where $3C$ is the cube with the same center under homothety by a factor of $3$. Then
\[\int_{C\cap\Sigma}|\nabla f|^2\le \|f\|_{L^2(\Sigma)}\|g\|_{L^2(B_2(p))}+\left(\epsilon \int_{(3C-C)\cap \Sigma}|\nabla f|^2+\frac{1}{\epsilon}\int_{(3C-C)\cap \Sigma}|f|^2\right).\]
Now choose $\epsilon=3^{-n}/3$. Cover $\RR^n$ by parallel cubes with no common interior and side length same as $C$, and also consider the $3$-scale homothety of each cube in the covering. We deduce that each point is covered by at most $3^ n$ times. Add the above inequality for each cube in this covering, we then get
\[\int_{\Sigma}|\nabla f|^2\le 3\|f\|_{L^2(\Sigma)}\|g\|_{L^2(\Sigma)}+3/\epsilon \|f\|_{L^2(\Sigma)},\]
hence $f\in W^{1,2}(\Sigma)$.
\end{proof}

Suppose $\ind(\Sigma)=I$ is finite. By Proposition \ref{generalizedfischer}, there exist $I$ $W^{1,2}$ smooth eigenfunctions $\varphi_1,\ldots,\varphi_I$ of the Jacobi operator $\Delta+|A|^2$. Consider the linear system on $\omega$
\begin{equation}\label{linearsystem}
\int_\Sigma \varphi_1 f_{\omega,ij}=0, \int_\Sigma \varphi_2 f_{\omega,ij}=0,\cdots,\int_\Sigma \varphi_I f_{\omega,ij}=0,
\end{equation}
where $i,j$ run through $1\le i<j\le n$.

Denote by $l$ the dimension of the space of $L^2$ harmonic $1$-forms on $\Sigma$. Now the linear system (\ref{linearsystem}) has $I\cdot \frac{n(n-1)}{2}$ equations. If $l>I\cdot \frac{n(n-1)}{2}$ then there exists at least $l-I\cdot \frac{n(n-1)}{2}$ linearly independent harmonic $1$-forms for which (\ref{linearsystem}) is satisfied by $f_{\omega, ij}$, for each pair of $i,j$ with $1\le i<j\le n$.  For each such $\omega$, by proposition \ref{generalizedfischer}, $Q(f_{\omega,ij},f_{\omega,ij})\ge 0$ for each pair of $1\le i<j\le n$. On the other hand,
\[
\begin{split}
    \sum &_{1\le i<j\le n}Q(f_{\omega,ij},f_{\omega,ij})=-\sum_{1\le i<j\le n}\int_{\Sigma}f_{\omega,ij}(\Delta+|A|^2)f_{\omega,ij}\\
                        &=2\int_\Sigma \sum_{1\le i<j \le n} \left(\bangle{\bar{V_i},\nu}\bangle{V_j,\xi}-\bangle{\bar{V_j},\nu}\bangle{V_i,\xi}\right)\left(\bangle{\nabla_{S(V_i)}\xi,V_j}-\bangle{\nabla_{S(V_j)}\xi,V_i}\right)\\
                        &=2\int_\Sigma\sum_{1\le i,j \le n}\left(\bangle{\bar{V_i},\nu}\bangle{V_j,\xi}-\bangle{\bar{V_j},\nu}\bangle{V_i,\xi}\right)\left(\bangle{\nabla_{S(V_i)}\xi,V_j}-\bangle{\nabla_{S(V_j)}\xi,V_i}\right)\\
                        &=2\int_\Sigma\sum_{i,j}\left(\bangle{\bar{V_i},\nu}\bangle{V_j,\xi}\bangle{\nabla_{S(V_i)}\xi,V_j}-\bangle{\bar{V_j},\nu}\bangle{V_i,\xi}\bangle{\nabla_{S(V_i)}\xi,V_j}\right)\\
\end{split}
\]
For the first summand,
\[
\begin{split}
    \sum_{i,j}&\bangle{\bar{V_i},\nu}\bangle{V_j,\xi}\bangle{\nabla_{S(V_i)}\xi,V_j}=\sum_i\bangle{\bar{V_i},\nu}\sum_j\bangle{\bar{V_j},\xi}\bangle{\nabla_{S(V_i)}\xi,\bar{V_j}}\\
              &=\sum_i \bangle{\bar{V_i},\nu}\bangle{\xi,\nabla_{S(V_i)}\xi}\\
              &=\frac{1}{2}\sum_i \bangle{\bar{V_i},\nu}\bangle{S(V_i),\nabla|\xi|^2}\\
              &=\frac{1}{2}\sum_i \bangle{\bar{V_i},\nu}\bangle{V_i,S(\nabla|\xi|^2)}\\
              &=\frac{1}{2} \bangle{\nu,S(\nabla|\xi|^2)}=0.
\end{split}
\]
Also
\[
\begin{split}
    \sum_{i,j}&\bangle{\bar{V_j},\nu}\bangle{V_i,\xi}\bangle{\nabla_{S(V_i)}\xi,V_j}=\sum_i \sum_j \bangle{\bar{V_j},\nu}\bangle{\nabla_{S(V_i)}\xi,\bar{V_j}}\bangle{V_i,\xi}\\
              &=\sum_i \bangle{\nu,\nabla_{S(V_i)}\xi}\bangle{V_i,\xi}=0.
\end{split}
\]

Therefore each $Q(f_{\omega,ij},f_{\omega,ij})$ is equal to zero. By proposition \ref{generalizedfischer}, $f_{\omega,ij}$ is in the kernel of Jacobi operator. To conclude the proof of Theorem $\ref{rough}$, we prove the $l-\frac{n(n-1)}{2}I$ linearly independent harmonic $1$-forms generate at least $\frac{2}{n(n-1)}l-I$ linearly independent functions $f_{\omega,ij}$. Then $\nul(\Sigma)\ge \frac{2}{n(n-1)}l-I\ge \frac{2}{n(n-1)}(\# \textrm{ends}+b_1(M)-1)-\ind(\Sigma)$. In fact, we have:

\begin{prop}\label{linearalgebra}
Let $\sH$ be an $h$ dimensional subspace of $L^2$ harmonic $1$-forms on $\Sigma$. Then the set $\{f_{\omega,ij}:\omega \in \sH, 1\le i<j \le n\}$ has at least $\frac{2}{n(n-1)}h$ linearly independent $L^2$ smooth functions on $\Sigma$.
\end{prop}

\begin{proof}
Define a map $\sF:\sH\rightarrow \oplus_{i=1}^{n(n-1)/2}C^{\infty}(M)$, $\omega \mapsto (f_{\omega,ij}: 1\le i<j \le n)$. We will prove that $\sF$ is injective. Suppose $\omega$ is a $L^2$ harmonic $1$-form such that $\sF(\omega)=0$. That is, $f_{\omega,ij}=0$ or $\bangle{\bar{V_i},\nu}\bangle{V_j,\omega}=\bangle{\bar{V_j},\nu}\bangle{V_i,\omega}$ for each pair $1\le i<j \le n$. Then $\bangle{V_i,\omega}=c\bangle{\bar{V_i},\nu}$ for some constant $c$. Since $\bar{V_1},\ldots,\bar{V_n}$ is an orthonormal basis for $\RR^n$, $\bangle{V,\omega}=c\bangle{\bar{V},\nu}$ for each parallel vector field $\bar{V}$ in $\RR^n$ and its projection $V$ on $\Sigma$. In particular, at a point $p\in \Sigma$, choose $\bar{V_1}=\nu(p)$ and $V_2,\ldots,V_n$ be a basis for $T_p \Sigma$, we get $c=0$ and $\omega(p)=0$.

Denote by $p_{ij}$ the projection of $\oplus_{i=1}^{n(n-1)/2} C^\infty(\Sigma)$ onto the $ij$-th component. Since $\sum_{1\le i<j\le n}\dim(p_{ij}(\sF(\sH))) \ge \dim(\sF(\sH))$, at least one pair $(i,j)$ satisfies
\[\dim(p_{ij}(\sF(\sH))) \ge \frac{2}{n(n-1)}\dim (\sF(\sH))=\frac{2}{n(n-1)}h.\]
For this particular $(i,j)$, the space of functions spanned by $\{f_{\omega,ij}:\omega \in \sH\}$ are at least $\frac{2}{n(n-1)}h$ dimensional.
\end{proof}

\begin{rema}
Let us look closer at the equality case in the proof of Theorem \ref{rough}. For any harmonic $1$-form $\omega$ with $f_{\omega,ij}$, $1\le i<j \le n$, in the kernel of Jacobi operator, we have $0=-\frac{1}{2}(\Delta + |A^2|)f_{\omega,ij}=\bangle{\nabla_{S(V_i)}\omega,V_j}-\bangle{\nabla_{S(V_j)}\omega,V_i}$. Locally, every $\omega$ can be written as $d\phi$ for some smooth harmonic function $\phi$. Then $\bangle{\nabla_{S(V_i)}\omega,V_j}=\bangle{\nabla_{S(V_j)}\omega,V_i}$ is equivalent to $\Hess \phi(S(V_i),V_j)=\Hess\phi (S(V_j),V_i)$. Since $\{V_i\}$ is a basis for $T\Sigma$, we conclude that $\Hess\phi (S(X),Y)=\Hess\phi (S(Y),X)$ for every pair of tangent vectors $X,Y$. Now taking a local orthonormal frame of principal vectors on $\Sigma$, we see that the above condition is equivalent to $\Hess \phi$ being diagonalized by principal vectors of $\Sigma$. We are able to bound the dimension of the space of such functions $\phi$ when there is a point on $\Sigma$ where all principal curvatures are distinct.
\end{rema}

\section{Rigidity case}
We prove that when $\Sigma^{n-1}$ in $\RR^n$ satisfies that there is a point where all principal curvatures are distinct, the space of $L^2$ harmonic $1$-forms on $\Sigma$ satisfying $(\Delta+|A|^2)f_{\omega,ij}=0$, for each pair of $(i,j)$, is at most $2n-3$ dimensional. 

\begin{prop}
Let $\Sigma^{n-1}$ be a connected minimal submanifold of an analytic manifold $N^{n}$. Suppose that at one point $p$ on $\Sigma$, all the principal curvatures of $\Sigma$ are different. Then the dimension of the function space
\[\{\phi: \Delta\phi =0, \Hess\phi (S(X),Y)=\Hess \phi (S(Y),X),\textrm{ for all vector fields $X,Y$}\}\]
is at most $2n-2$.
\end{prop}

\begin{proof}
Take an orthonormal frame in a small neighborhood of the point $p$ consisting of principal vectors $e_1,\ldots,e_{n-1}$ with corresponding principal curvatures $\lambda_1,\ldots,\lambda_{n-1}$ (all distinct), respectively. Then for any function $\phi$ with $\Hess \phi(S(X),Y)=\Hess \phi(S(Y),X)$, letting $X=e_i$ and $Y=e_j$ for $i\ne j$, we get $\Hess \phi (e_i,e_j)=0$. Also $\Delta \phi=0$ implies $\sum_i\Hess \phi(e_i,e_i)=0$. Now $\Sigma$ is an analytic manifold since it is a  minimal hypersurface of an analytic manifold. By the unique extension theorem, any harmonic function is uniquely determined by all its derivatives at one point $p$. We prove that if a harmonic function $\phi$ satisfies the extra condition that $\Hess \phi$ commutes with the shape operator $S$, all the covariant derivatives $\nabla^j \phi (p)$ are uniquely determined by $\phi(p),\nabla_{e_1}\phi (p),\ldots, \nabla_{e_{n-1}}\phi(p), \nabla^2_{e_1,e_1}\phi (p),\ldots,\nabla^2_{e_{n-2},e_{n-2}}\phi (p)$, so the dimension of all such functions is at most $2(n-1)$.

Let us prove that if $\Delta\phi=0$ and $\phi(p)=\nabla_{e_1}\phi(p)=\ldots=\nabla_{e_{n-1}}\phi(0)=\nabla_{e_1,e_1}\phi(p)=\ldots,\nabla_{e_{n-2},e_{n-2}}\phi(p)=0$ then all derivatives $\nabla^j_{e_{i_1},\ldots,e_{i_j}}\phi(p)=0$. We'll proceed by induction on $j$. The cases of $j\le 2$ are given as assumptions. Now suppose $j>2$, and that any covariant derivatives of $\phi$ with order less than or equal to $j-1$ are zero. Consider a covariant derivative $\nabla^j_{e_{i_1},e_{i_2},\ldots,e_{i_j}}\phi$. We separate two cases.

\begin{itemize}
    \item[Case 1] Not all of $i_1,\ldots,i_j$'s are equal. Then after switching the order of taking derivatives finitely many times, we will get an expression of $\nabla^j_{e_{i_1'},\ldots,e_{i_j'}}$ with $i_{j-1}'\ne i_j'$. Every time we switch two consecutive indices $i_\alpha,i_{\alpha+1}$, the difference we get is a curvature term depending linearly on lower order derivatives of $\phi$ at $p$. By assumption all lower order derivatives of $\phi$ at $p$ are zero. On the other hand, since $i_{j-1}'\ne i_{j}'$, $\nabla^2_{e_{i_{j-1}'},e_{i_{j}'}}\phi=0$. Therefore, in this case, $\nabla^j_{e_{i_1},e_{i_2},\ldots,e_{i_j}}\phi=0$.

    \item[Case 2] $i_1=i_2=\ldots=i_j$ are all equal. Without loss of generality we may assume $i_1=1$. Since $\Delta\phi=0$, $\nabla^2_{e_1,e_1}\phi=-\sum_{i=2}^n \nabla^2_{e_i,e_i}\phi$. Therefore $\nabla^j_{e_1,\ldots,e_1}\phi=-\nabla^j_{e_1,\ldots,e_1,e_2,e_2}\phi-\ldots-\nabla^j_{e_1,\ldots,e_1,e_j,e_j}\phi$. From case 1 we know $\nabla^j_{e_1,\ldots,e_1}\phi=0$.

\end{itemize}

\end{proof}

\begin{rema}
	Note that by writing $\omega=d\phi$, one increases the dimension of the space of harmonic $1$-forms by one($\phi$ and $\phi+C$ gives the same $\omega$, for every constant $C$). Therefore we conclude that the space of harmonic $1$-forms $\omega$ satisfying
	\[\bangle{\nabla_{S(V_i)}\omega, V_j}=\bangle{\nabla_{S(V_j)}\omega,V_i},\quad \forall 1\le i\le j\le n\]
	is at most $2n-3$, under the assumption of previous proposition. When $n=3$, this fact has also been utilized by A. Ros, see \cite{ros2006one}. We generalize it for all dimensions.
\end{rema}

The next geometric theorem shows the assumptions of the previous proposition holds for general minimal hypersurfaces in $\RR^4$.

\begin{theo}\label{rigidity}
Suppose $\Sigma^3\subset \RR^4$ is a connected complete minimal hypersurface, with the property that at each point there are two equal principal curvatures. Then $\Sigma$ is either a hyperplane or a catenoid.
\end{theo}

\begin{proof}
If the principal curvature at every point is $0$, then $\Sigma$ is a hyperplane. We assume that there is an open subset $U$ of $\Sigma$ such that principal curvatures of $\Sigma$ in $U$ are given by $\lambda,\lambda,-2\lambda$ for some nonzero $\lambda$. Denote $\bar{\nabla}$ the connection in $\RR^4$, and $\nabla$ the connection on $\Sigma$. Choose an orthonormal frame $\{e_1,e_2,e_3\}$ locally in $U$, and let $N$ be its unit normal vector in $\RR^4$, such that $\bar{\nabla}_{e_1}N=\lambda e_1,\bar{\nabla}_{e_2}N=\lambda e_2,\bar{\nabla}_{e_3}N=-2\lambda e_3$.

We first prove that $\Span\{e_1,e_2\}$ is an integrable distribution. For this, let's show $[e_1,e_2]$ is also a principal vector with curvature $\lambda$.

By the Gauss equation, we have
\begin{equation}\label{eq1}
\begin{split}
\bar{\nabla}_{[e_1,e_2]} N &= \bar{\nabla}_{e_1}\bar{\nabla}_{e_2}N - \bar{\nabla}_{e_2}\bar{\nabla}_{e_1} N\\
                       &= e_1(\lambda)e_2-e_2(\lambda)e_1+\lambda [e_1,e_2].
\end{split}
\end{equation}

Suppose $[e_1,e_2]=a_1e_1+a_2e_2+a_3e_3$. Then we have $\bar{\nabla}_{[e_1,e_2]}N=\sum a_i\bar{\nabla}_{e_i}N=a_1\lambda e_1+a_2\lambda e_2-2a_3\lambda e_3$. On the other hand, by (\ref{eq1}), we have $\bar{\nabla}_{[e_1,e_2]}N=\lambda (a_1 e_1+a_2e_2+a_3e_3)+e_1(\lambda)e_2-e_2(\lambda)e_1$. Therefore we see that $a_3=0$ and $e_1(\lambda)=e_2(\lambda)=0$.

Denote by $\Gamma$ to be the integral submanifold of the distribution spanned by $\{e_1,e_2\}$. From the above we also see that $\lambda$ is constant along $\Gamma$. We next prove that $\Gamma$ is part of a sphere.

To see this, we first note that $\bar{\nabla}_{e_1}e_3$ has no component in $e_3$ and $N$, and $\bar{\nabla}_{e_3}e_1$ has no component in $e_1$ and $N$. Therefore we may assume
\begin{align}
\bar{\nabla}_{e_1}e_3&=ae_1+be_2\\
\bar{\nabla}_{e_3}e_1&=ce_2+de_3.
\end{align}
Then by the Gauss equation,
\begin{equation*}
\begin{split}
\bar{\nabla}_{[e_1,e_3]}N&=\bar{\nabla}_{e_1}\bar{\nabla}_{e_3}N-\bar{\nabla}_{e_3}\bar{\nabla}_{e_1}N\\
                   &=-2\lambda\bar{\nabla}_{e_1}e_3-e_3(\lambda)e_1-\lambda\bar{\nabla}_{e_3}e_1\\
                   &=-2\lambda(ae_1+be_2)-e_3(\lambda)e_1-\lambda(ce_2+de_3).
\end{split}
\end{equation*}
On the other hand, we have
\begin{equation*}
\begin{split}
\bar{\nabla}_{[e_1,e_3]}N&=\bar{\nabla}_{\bar{\nabla}_{e_1}e_3-\bar{\nabla}_{e_3}e_1}N\\
                   &=a\lambda e_1+b\lambda e_2-c\lambda e_2+2d\lambda e_3.
\end{split}
\end{equation*}
Comparing coefficients, we obtain $b=d=0, a=-\frac{e_3(\lambda)}{3\lambda}$. That is,
\[\bar{\nabla}_{e_1}e_3=-\frac{e_3(\lambda)}{3\lambda}e_1,\quad \bangle{\bar{\nabla}_{e_3}e_1,e_3}=0.\]
For similar reasons we also have
\[\bar{\nabla}_{e_2}e_3=-\frac{e_3(\lambda)}{3\lambda}e_2,\quad \bangle{\bar{\nabla}_{e_3}e_2,e_3}=0.\]

Let $\alpha=-\frac{e_3(\lambda)}{\lambda}$. Then we calculate $\bar{\nabla}_{[e_1,e_2]}e_3$. Again by the Gauss equation, we see $\bar{\nabla}_{[e_1,e_2]}e_3=\alpha [e_1,e_2]+e_1(\alpha)e_2-e_2(\alpha)e_1$. However since $[e_1,e_2]$ is in $\Span\{e_1,e_2\}$, we have $\bar{\nabla}_{[e_1,e_2]}e_3=\alpha[e_1,e_2]$. Therefore $e_1(\alpha)=e_2(\alpha)=0$, so that $\alpha$ is constant along $\Gamma$.

We now have $\bar{\nabla}_{e_i}N=\lambda e_i$ and $\bar{\nabla}_{e_i}e_3=\alpha e_i$ for some constant $\lambda,\alpha$ along $\Gamma$. Viewing $\Gamma$ as a vector valued function $X$, we see that $X-\frac{1}{\lambda}N$ and $X-\frac{1}{\alpha}e_3$ are both constant vectors (when $\alpha=0$ the second conclusion is $X$ lies on a plane). Hence $X$ lies on the intersection of two 3-spheres (when $\alpha\ne 0$), or the intersection of a 3-sphere and a hyperplane (when $\alpha=0$). In either case, $\Gamma$ is a part of a 2-sphere.

The above proves a foliation structure of $\Sigma$ by spheres $\Gamma$. By a result of Jagy (corollary of section 4 in \cite{jagy1991}), a connected minimal hypersurface of $\RR^4$ with an open set foliated by spheres possesses $SO(3)$ symmetry. Hence $\Sigma$ is a $3$-dimensional catenoid.
\end{proof}

\begin{rema}
The same proof directly gives the higher dimensional analogue of theorem \ref{rigidity}. Namely, if $\Sigma^{n-1}$ is a connected minimal hypersurface in $\RR^n$ with the property that at every point on $\Sigma$ there is a principal curvature with multiplicity $n-2$, then $\Sigma$ is either a hyperplane or a higher dimensional catenoid.
\end{rema}

\begin{rema}
	After completing this work, the author is informed by the referee that theorem \ref{rigidity} was known to do Carmo and Dajczer, see \cite{do1983rotation}. Here we provide a different proof of the same statement.
\end{rema}

To finish the proof of theorem \ref{precise}, recall that $l=\#\textrm{ends}+b_1(\bar{\Sigma})-1$ and that $l-\frac{n(n-1)}{2}\Index (\Sigma)$ is bounded by the dimensional of harmonic one forms $\omega$ satisfying 
\[\bangle{\nabla_{S(V_i)}\omega,V_j}=\bangle{\nabla_{S(V_j)}\omega,V_i},\quad \forall 1\le i\le j\le n.\]
 For $\Sigma^{n-1}\subset \RR^n$, if there is a point on $\Sigma$ where all the principal curvatures are distinct, then we may apply the above theorems to conclude that the dimension of such harmonic one forms is bounded by $2n-3$, and theorem \ref{precise} follows. When $n=4$, we either have that $\Sigma$ is a hyperplane, or the $3$-dimensional catenoid, or that $\Sigma$ has a point with $3$ distinct principal curvatures. In the first case $l=0$ and $\Index(\Sigma)=0$ (the hyperplane is stable). In the second case $l=1$ and $\Index(\Sigma)=1$ (\cite{tam2009stability}). In the third case, $l-6\Index(\Sigma)\le 5$, therefore $\Index(\Sigma) \ge \frac{1}{6}(\#\textrm{ends}+b_1(\Sigma))-1$.

\section{The space of index $1$ minimal hypersurfaces in $\RR^4$}
In this section we consider the space of index $1$ minimal hypersurfaces in $\RR^4$ with Euclidean volume growth. For minimal hypersurfaces with Euclidean volume growth in $\RR^n$, $4\le n\le 7$, finite index is equivalent to finite total curvature. Moreover, theorem \ref{precise} implies a control of the volume growth rate in terms of index. Consider the set
\begin{equation*}
    \begin{split}
        \cS=\{&\Sigma^3\subset \RR^4: \Sigma \textrm{ is a complete connected embedded  oriented minimal hypersurface}\\
              &\textrm{ with index $1$ and Euclidean volume growth, } |A_\Sigma|(0)=\max|A_\Sigma|=1.\}
    \end{split}
\end{equation*}
Then the volume growth rate of every surface in $\cS$ is uniformly bounded. That is, for every $\Sigma\in \cS$, $R>0$,
\[\frac{\Vol(\Sigma\cap B_R(0))}{\omega_3 R^3}\le \eta.\]
For example, we may take $\eta=15$.

Let us show that the space $\cS$ is compact in the smooth topology. Take a sequence $\Sigma_j$ in $\cS$. We first observe that up to a subsequence (which we also denote by $\Sigma_j$), there are two modes of convergence. The first is by the fact that the curvature of $\Sigma_j$ is uniformly bounded. Therefore by Arzela-Ascoli, there is a subsequence converging locally graphically in $C^{1,\alpha}$ to some $\Sigma$. From standard minimal surface theory, this also implies the convergence is locally smooth. The second mode of convergence is that, since we have a uniform density bound, the varifolds determined by $\Sigma_j$ have uniformly bounded local mass. By Allard's compactness theorem, a subsequence converges as varifolds to some $\Sigma'$. By the constancy theorem, $\Sigma'$ is supported on $\Sigma$. As a result, we get that $\Sigma_j$ converges to $\Sigma$ both locally smoothly and in the varifold sense. Now the varifold convergence implies that the second variation of $\Sigma_j$ converges to $\Sigma$. In particular, the index of $\Sigma$ cannot be larger than $1$ (otherwise for large $j$, there will be at least two negative eigenfunctions for the Jacobi operator on $\Sigma_j$). However, from smooth convergence we know $|A_\Sigma|(0)=|A_{\Sigma_j}|(0)=1$, hence by \cite{shen1998stable}, $\Sigma$ cannot be stable. Therefore we conclude that $\Sigma$ has index $1$.

It remains to prove that $\Sigma$ is connected. The argument we use here is similar to \cite{chodosh2015minimal}. The following observation of White asserts that rapid curvature decay implies simple topology, namely
\begin{prop}[\cite{white1987curvature}]\label{brian}
Let $\Sigma^{n-1}$ be a minimal hypersurface in Euclidean space. Assume for all $x\in B_R(0)^c$, $|A_\Sigma|(x)\cdot |x|\le \frac{1}{4}$, and $\Sigma$ intersects $\partial B_R(0)$ transversely on $k$ connected components, each one diffeomorphic to $S^{n-2}$. Then each component of $\Sigma-B_R(0)$ is diffeomorphic to $S^{n-2}\times [0,1)$.
\end{prop}

We briefly mention the proof of this proposition. Under the curvature condition, $|x|^2$ is a Morse function with no critical point in $\Sigma-B_R(0)$. Therefore by Morse theory each connected component of it is diffeomorphic to $S^{n-2}\times [0,1)$.

Now go back to the proof of our compactness theorem. Each $\Sigma_j$ and $\Sigma$ have finite index and Euclidean volume growth, hence they are regular at infinity. So there are constants $R_j$ such that $\Sigma_j$ intersects $\partial B_{R_j}(0)$ transversely and $|A_{\Sigma_j}|\cdot |x|\le \frac{1}{4}$, for $x\in \Sigma-B_{R_j}(0)$. Assume also this $R_j$ is the least possible choice.
\begin{claim}
$\{R_j\}$ is bounded.
\end{claim}

Assuming the claim, the connectedness of $\Sigma$ follows. Indeed, suppose $R_j<R$. Then for each $j$, by Proposition \ref{brian}, $\Sigma_j\cap B_R(0)$ is connected. Now the varifold convergence of $\Sigma_j\rightarrow \Sigma$ implies Hausdorff convergence in compact set. Therefore $\Sigma\cap B_R(0)$ is connected. From this and Proposition \ref{brian} we see that $\Sigma$ is connected.

Let us now prove the claim. Suppose the contrary. Then by taking a further subsequence (which we still denote by $\Sigma_j$), $R_j\rightarrow \infty$. Consider the rescaled sequence $\bar{\Sigma}_j=\frac{1}{R_j}\Sigma$. The sequence $\bar{\Sigma}_j$ has the same density at infinity as $\Sigma_j$, hence by Allard's compactness theorem they converge, up to a subsequence, to some varifold $\bar{\Sigma}$. By the choice of $R_j$ we see that the curvature estimate $|A_{\bar{\Sigma}_j}|\cdot |x|\le \frac{1}{4}$ holds for $x\in \bar{\Sigma}_j-B_1(0)$. By Proposition \ref{brian}, each $\bar{\Sigma}_j\cap B_1(0)$ is connected. Therefore $\bar{\Sigma}$ is connected. Now the curvature of $\Sigma_j'$ blows up at $0$, so the convergence cannot be smooth at $\{0\}$. Since $\bar{\Sigma}_j$ has index $1$, the surface $\bar{\Sigma}$ is regular everywhere, and the convergence is not smooth at no more than $1$ point. So $\{0\}$ is the unique point where the convergence $\bar{\Sigma}_j\rightarrow \bar{\Sigma}$ is not smooth. By Allard's theorem the convergence cannot be of multiplicity $1$. Note also that by \cite{tysk1989finiteness}, the total curvature of each $\Sigma_j$ is also uniformly bounded. Now that the convergence is at least 2-sheeted, we may use an argument in \cite{bensharp2015compact} to produce a positive Jacobi field on $\bar{\Sigma}-\{0\}$ by taking the distance between two sheets then normalized properly. Then the uniform bound of total curvature of $\Sigma_j$ implies that, this Jacobi field can be extended over $\{0\}$. We refer the readers to \cite{bensharp2015compact} for a detailed explanation. This means that $\bar{\Sigma}$ is stable. Therefore $\bar{\Sigma}$ is a plane through $0$.

By the choice of $R_j$ there exists some $x_j\in \bar{\Sigma}_j \cap \partial B_1(0)$ such that $|A_{\bar{\Sigma}_j}|(x_j)=\frac{1}{4}$. Taking a subsequence of $x_j$ converging to some $x\in \bar{\Sigma}\cap \partial B_1(0)$, we get a contradiction, since $|A_{\bar{\Sigma}}(x)|=0$, and $\bar{\Sigma}_j\rightarrow \bar{\Sigma}$ smoothly near $x$. The claim is proved.

Theorem \ref{compactnessofindex1} roughly says that an index $1$ minimal hypersurface in $\RR^4$ cannot have two necks that are far away, in constrast to the phenomenon described in Example \ref{costa}. This, together with the following corollary, can be viewed as evidence that the $3$ dimensional catenoid is the unique embedded index $1$ minimal hypersurface with Euclidean volume growth in $\RR^4$.

\begin{coro}
There exists a constant $R$ such that for any minimal hypersurface $\Sigma^3$ in $\RR^4$ with index $1$ and Euclidean volume growth, such that $|A_\Sigma|(0)=\max |A_\Sigma|=1$, we have that $\Sigma-B_R(0)$ is the union of minimal graphs.
\end{coro}

\section{Finite diffeomorphism types of minimal hypersurfaces in $\RR^4$ with Euclidean volume growth and bounded index}

Recently, Chodosh-Ketover-Maximo proved the following finiteness diffeomorphism result of minimal hypersurfaces $\Sigma^{n-1}\subset \RR^n$, $4\le n\le 7$.

\begin{theo}[\cite{chodosh2015minimal}]
For $4\le n\le 7$, there is $N=N(n,I,\Lambda)\in \NN$ so that there are at most $N$ mutually non-diffeomorphic complete embedded minimal hypersurfaces $\Sigma^{n-1}\subset \RR^n$ with $\Index(\Sigma)\le I$ and $\Vol(\Sigma\cap B_R(0))\le \Lambda R^{n-1}$ for all $R>0$.
\end{theo}

Recall that a minimal hypersurface $\Sigma^{n-1}\subset\RR^n$ with Euclidean volume growth and finite index must be regular at infinity. By the monotonicity formula, the volume growth rate $\lim_{R\rightarrow \infty}\Vol(\Sigma\cap B_R(0))/\omega_{n-1} R^{n-1}$ is equal to the number of ends. When $n=4$, theorem \ref{precise} provides an upper bound of the number of ends in terms of the index. In fact, a minimal hypersurfaces $\Sigma^3\subset\RR^4$ with Euclidean volume growth and index $I$ must satisfy $\Vol(\Sigma\cap B_R(0))\le (6I+7)\omega_{n-1} R^{n-1}$. As a result, we have

\begin{theo}
There exists $N=N(I)$ such that there are at most $N$ mutually non-diffeomorphic complete embedded minimal hypersurfaces $\Sigma^3$ in $\RR^4$ with Euclidean volume growth and $\ind(\Sigma)\le I$.
\end{theo}

\appendix
\section{Minimal hypersurfaces with finite total curvature}

	In the appendix we give a brief explanation of proposition 2.4. Namely, we prove that immersed minimal hypersurfaces with finite total curvature in $\RR^n$ is regular at infinity. The proof we include here is a generalization of \cite{tysk1989finiteness} and \cite{anderson1984compactification}.  We refer the readers to their original papers for more details.
	
	\begin{proof}
	Let $\Sigma^{n-1}\subset \RR^n$ be a minimal hypersurface with finite total curvature, i.e. $\int |A|^{n-1}< \infty$. We prove that $\Sigma$ has finitely many embedded ends, each of which is a graph over some affine plane of some function $u$ 
	\[|x|^{n-3}|u(x)|+|x|^{n-2}|\nabla u(x)|+|x|^{n-1}|\nabla^2 u(x)|\le C.\]
	The proof can be divided into several steps as follows.
	
	\begin{itemize}
    \item[Step 1] A $\epsilon$ regularity theorem for $|A|$. 
    
    Consider an annulus $A(R,2R)=B(0,2R)-B(0,R)$. We have the following: 
    \begin{prop}
    	There is an $\epsilon_0>0$ such that if $\int_{A(R,2R)\cap \Sigma}|A|^{n-1}<\epsilon_0$ then 
    	\[\sup_{x\in \partial B(0,R)\cap \Sigma}|A|^2(x)\le \frac{1}{R^2}\mu\left(\int_{A(R,2R)\cap \Sigma}|A|^{n-1}.\right)\]
    	Here $\mu$ is a continuous function satisfying $\mu(\epsilon)\rightarrow 0$ as $\epsilon\rightarrow 0$.
    \end{prop}
	This proposition is a direct consequence of the Simons' equality and the scaling invariance of $\int |A|^{n-1}$. As a consequence, since the total curvature is finite, we deduce that 
	\[|A|(x)\cdot |x|\rightarrow 0, \quad \textrm{as $|x|\rightarrow \infty$}.\]
	
	\item[Step 2] Euclidean volume growth.
	
	Recall that by \cite{tysk1989finiteness} the Morse index of $\Sigma$ is bounded from above by a constant times the total curvature, and in particular, the index of $\Sigma$ is finite. We then conclude by \cite{li2002minimal} that $\Sigma$ has finitely many ends. We study the asymptotic behavior of each end.
	
	By the Morse-theoretic argument in \cite{white1987curvature}, we see that for $R>R_0$, the distance function $r=\dist(\cdot, 0)$ is a Morse function with no critical point. Take $R$ sufficiently large such that $\Sigma\cap \partial B_R(0)$ transversely, and each connected component of $\Sigma\cap \partial B_R(0)$ corresponds to an end of $\Sigma$. We first prove that $\Sigma-B_{R_0}(0)$ is a collection of finitely many ends, each diffeomorphic to $S^{n-2}\times [0,\infty)$. In particular, this shows that $\Sigma$ is properly immersed with finite topology. Let us look at one end $V$.
	
	Let $V_{r_0}=\partial B_{r_0}(0)\cap V$ to be level set of $r$ on $V$. Since
	\[(\nabla^{\Sigma })^2  r(X,Y)=(\nabla^{\RR^n})^2 r(X,Y)-\bangle{A(X,Y),\nabla^{\RR^n}r },\]
	and that $|A|(r)\le \frac{\mu_1(r)}{r}$, where $\mu_1$ is the continuous function obtained in step 1, that converges to zero as $r$ tends to infinity, we conclude that
	\[|A_{V_r}-\frac{1}{r}I|\le \frac{\mu_1(r)}{r},\]
	where $A_{V_r}$ is the second fundamental form of $V_r$ inside $V$. By the Gauss equation we conclude that
	\begin{equation}\label{curvature}
	|K_{V_r}(x,P)-\frac{1}{r^2}|\le\frac{2\mu_1(r)}{r}
	\end{equation}
	for every point $x$ on $V_r$ and every tangential two-plane $P$ in $T_xV_r$, where $K_{V_r}$ is the sectional curvature. Since $\mu_1(r)\rightarrow 0$ as $r\rightarrow \infty$, $V_r$ is diffeomorphic to the standard sphere of dimension $n-2$. Combining this with the Morse-theoretic argument, we conclude that each end of $\Sigma$ is diffeomorphic to $S^{n-1}\times [0,1)$. Therefore $\Sigma$ is properly immersed with finite topology. 
	
	Let us prove that $\Sigma$ has Euclidean volume growth. By the curvature condition \ref{curvature} and the standard volume comparision, we see that for each level surface $V_r$,
	\[1-2\mu_1(r)<\frac{\Vol(V_r)}{\omega_{n-2}r^{n-2}}-1<1+2\mu_1(r).\]
	Hence by the co-area formula we conclude that $V$ has Euclidean volume growth, i.e. 
	\[\lim_{R\rightarrow \infty}\frac{\Vol(V\cap B_R(0))}{\omega_{n-1}R^{n-1}}=1.\]
	
	Since $\Sigma$ has only finitely many ends at infinity, $\Sigma$ also has Euclidean volume growth.
	
	\item[Step 3] Tangent plane at infinity.
	
	Having Euclidean volume growth, consider the varifold limit $\Sigma_\infty$ of rescaled surfaces $\frac{1}{r}\Sigma$. By slight abuse of notation we also use $\Sigma_\infty$ to denote the support of this varifold. For any positive $\delta>0$, $\Sigma_\infty-B_\delta(0)$ is the limit of minimal hypersurfaces whose $|A|$ converges to $0$ uniformly. Therefore $\Sigma_\infty$ is totally geodesic outside of $0$, i.e. $\Sigma_\infty$ is a union of hyperplanes through $0$. 
	
	Choose a radius $r_0$ sufficiently large such that for any $r>r_0$ the curvature condition \ref{curvature} holds. Consider an end $V\setminus B_{r_0}(0)$ and its rescalings $\{\frac{1}{R}(V\setminus B_{r_0}(0))\}$. Previously we know that as $R\rightarrow \infty$ these rescaled ends converge subsequentially to $P\setminus \{0\}$ for some hyperplane $P$. By the curvature condition \ref{curvature} we know that 
	\[\lim_{R\rightarrow \infty}\frac{\Vol((V\setminus B_{r_0}(0))\cap B_R(0))}{\omega_{n-1}R^{n-1}}=1.\]
	Since the varifold convergence does not increase the density at infinity, the convergence
	\[\frac{1}{R}(V\setminus B_{r_0}(0))\rightarrow P\setminus \{0\}\]
	is in multiplicity one.
	
	Next we point out that the plane $P$ does not depend on the choice of the subsequence in the convergence, that is, each end has a unique tangent plane at infinity. To see this, first note that since $\Sigma$ has finite index, the end $V\setminus B_{r_0}(0)$ is stable for sufficiently large radius $r_0$. Also since for each sequence of $R_i\rightarrow \infty$ there is a subsequence of $\{\frac{1}{R_i}(V\setminus B_{r_0}(0))\}$ converging to some hyperplane of multiplicity one, by Lemma 3 of \cite{tysk1989finiteness} we conclude that this limiting hyperplane is unique.
	
	This means that the each end of the original surface $\Sigma$ has a unique tangent plane at infinity.
	
	\item[Step 4] Regular at infinity.
	
	Since each end of $\Sigma$ converges to a hyperplane of multiplicity one, we deduce that outside some compact set, each end is a graph of bounded slope over the tangent plane at infinity. According to proposition 3 of \cite{schoen1983}, this end is regular at infinity.

	\end{itemize}

	\end{proof}

\bibliography{bib}
\bibliographystyle{amsalpha}

\end{document}